\documentclass[numbers,sort&compress]{elsarticle}

\usepackage{amsmath,mathtools,amsfonts, amssymb}
\usepackage{bbold, mathrsfs, amscd,latexsym,amsthm,}
\usepackage{tikz}\usetikzlibrary {arrows.meta,bending,math,shapes.geometric,patterns}
\usepackage{caption}

\usepackage{enumerate}
\usepackage{etoolbox}
\usepackage{stmaryrd, booktabs}
\usepackage{cases}
\usepackage{tabularx, colortbl}
\usepackage{parskip}
\usepackage{geometry}

\usepackage{color, xcolor}
\usepackage{graphicx}
\usepackage{lettrine}
\usepackage{setspace}
\usepackage{verbatim}

\usepackage{hyperref, url}
\usepackage{framed}

\usepackage{xcolor}
\usepackage{framed}
\definecolor{shadecolor}{RGB}{153,204,255}

\DeclarePairedDelimiter{\abs}{|}{|}
\DeclarePairedDelimiter{\aset}{\{}{\}}
\DeclarePairedDelimiter{\extclass}{\llbracket}{\rrbracket}

\newcommand{\iiome}{_{i\in \omega}}  
\newcommand{\riome}{_{r\in \omega}}  
\newcommand{\siome}{_{s\in \omega}} 
\newcommand{\tiome}{_{t\in \omega}}  
\newcommand{\jiome}{_{j\in \omega}}  

\theoremstyle{plain}
\newtheorem{theorem}{Theorem}[section]
\newtheorem{lemma}[theorem]{Lemma}
\newtheorem{corollary}[theorem]{Corollary}
\newtheorem{proposition}[theorem]{Proposition}

\newtheorem{question}{Question}
\theoremstyle{definition}
\newtheorem{definition}[theorem]{Definition}

\theoremstyle{remark}

\newtheorem{remark}[theorem]{Remark}

\newcommand{\Nat}{\mathbb{N}}

\newcommand{\Xs}{\mathsf X}

\newcommand{\concat}{\, ^\smallfrown\, }

\newcommand{\ml}{Martin-L\"{o}f }

\newcommand{\ie}{i.e.\ }
\newcommand{\ien}{i.e.}

\newcommand{\lce}{left-c.e.\ }
\newcommand{\rce}{right-c.e.\ }
\newcommand{\dce}{d.c.e.\ }
\newcommand{\ca}{c.a.\ }

\newcommand{\lcen}{left-c.e.}
\newcommand{\rcen}{right-c.e.} 
\newcommand{\dcen}{d.c.e.}
\newcommand{\can}{c.a.}

\begin{document}
\title{Speedability of computably approximable reals and their approximations}
\author[aff1]{George Barmpalias}
\ead{barmpalias@gmail.com}
\author[aff1]{Nan Fang\corref{cor1}}
\ead{fangnan@ios.ac.cn}
\author[aff2]{Wolfgang Merkle}
\ead{merkle@math.uni-heidelberg.de}
\author[aff2,aff3]{Ivan Titov}
\ead{me@ivantitov.de}

\affiliation[aff1]{organization={Key Laboratory of System Software (Chinese Academy of Sciences), Institute of Software, Chinese Academy of Sciences},
city={Beijing},
country={China}}
\affiliation[aff2]{organization={Institut für Informatik, Ruprecht-Karls-Universität Heidelberg},
country={Germany}}
\affiliation[aff3]{organization={Université de Bordeaux, CNRS, Bordeaux INP, LaBRI, UMR 5800, F-33400 Talence},
country={France}}
\cortext[cor1]{Corresponding author}

\begin{abstract}
An approximation of a real is a sequence of rational numbers that converges to the real. An approximation is left-c.e.\ if it is computable and nondecreasing and is d.c.e.\ if it is computable and has bounded variation. A real is computably approximable if it has some computable approximation, and left-c.e.\ and d.c.e.\ reals are defined accordingly.

An approximation $\{a_s\}_{s \in \omega}$  is \emph{speedable} if there exists a nondecreasing computable function $f$ such that the approximation $\{a_{f(s)}\}_{s \in \omega}$ converges in a certain formal sense faster than $\{a_s\}_{s \in \omega}$. This leads to various notions of speedability for reals, e.g., one may require for a computably approximable real that either all or some of its approximations of a specific type are speedable.

Merkle and Titov established the equivalence of several speedability notions for left-c.e.\ reals that are defined in terms of left-c.e.\ approximations. We extend these results to d.c.e.\ reals and d.c.e.\ approximations, and we prove that in this setting, being speedable is equivalent to not being Martin-L\"{o}f random. Finally,  we demonstrate that every computably approximable real has a computable approximation that is speedable.
\end{abstract}
\maketitle
\section{Introduction}\label{sec:introduce}
\subsection{Left-c.e.\ reals and Solovay reducibility}
An approximation of a real is a sequence of rational numbers that converges to that real. In what follows, unless explicitly stated otherwise, by real we always refer to a real in~$(0,1)$, the open unit interval, and all rational numbers occurring in an approximation are in~$[0,1]$, the closed unit interval.  An approximation~$\aset{a_s}\siome$ is~\emph{computable} if~$a_s$ can be computed for given~$s$. Restricted variants of computable approximations are defined as follows. A computable approximation is \emph{left-c.e.} if it is nondecreasing, \emph{right-c.e.}\ if it is nonincreasing, and \emph{d.c.e}\ if it has bounded variation, i.e., if the sum over the terms~$\abs{a_s - a_{s+1}}$ is finite. A real is \emph{computably approximable} if it has a computable approximation, a real is \lce if it has a \lce approximation, and \rce  and \dce reals are defined accordingly. 

\begin{remark}\label{remark:dyadic-approximation}
Every \ca real has a computable approximation that consists of pairwise distinct dyadic rationals, i.e., of rationals with a denominator of the form~$2^k$. Likewise, every left-c.e.\ real has a left-c.e.\ approximation that consists of pairwise distinct dyadic rationals, and a similar remark holds for \rce and for \dce reals. 
The easy proofs of these well-known facts are left to the reader. 

Recall that an \emph{order} is a nondecreasing and unbounded function from the set of natural numbers to itself. For further use, note that for any computable order~$f$,  given a \ca approximation~$\aset{a_s}\siome$, then~$\aset{a_{f(s)}}\siome$ is a \ca approximation of the same real. It is easy to check that this implication remains true when replacing the two occurrences of \ca either both by \dcen, or both by \lcen, or both by \rcen.
\end{remark}

The class of \lce reals is a central object of study in the field of algorithmic randomness. Solovay reducibility is used to compare such reals by the speed of convergence of their \lce approximations. The original definition of Solovay reducibility by Solovay~\cite*{solovaydraft} applies to all reals, but when restricted to \lce reals, the following equivalent definition due to Calude et al.~\cite*{SCALUDE2001125} is better suited and more relevant to our discussion (see also the monograph by Downey and Hirschfeldt~\cite*{rodenisbook}). For \lce reals $\alpha$ and $\beta$, $\alpha$ is \emph{Solovay reducible} to $\beta$, denoted $\alpha \le_S \beta$, if 
\begin{itemize}
\item[$(*)$] 
there exist a \lce approximation $\{a_s\}_{s \in \omega}$ of $\alpha$, a \lce approximation  $\{b_s\}_{s \in \omega}$ of $\beta$, and a constant $c$ such that $\alpha - a_s \leq c (\beta - b_s)$ for all $s$.
\end{itemize}
It is well-known that for \lce reals~$\alpha$ and~$\beta$, Condition $(*)$ is equivalent to the following Condition $(*')$, and is equivalent to Condition $(*'')$ in case~$\beta$ is not rational.
\begin{itemize}
	\item[$(*')$] For any \lce approximation $\{a'_s\}_{s \in \omega}$ of $\alpha$, there exists a \lce approximation $\{b'_s\}_{s \in \omega}$ of $\beta$ and a constant $c$ such that  $\alpha - a'_s \leq c (\beta - b'_s)$ for all $s$.
	\item[$(*'')$] For any \lce approximation $\{b''_s\}_{s \in \omega}$ of $\beta$, there exists a \lce approximation $\{a''_s\}_{s \in \omega}$ of $\alpha$ and a constant $c$ such that  $\alpha - a''_s \leq c (\beta - b''_s)$ for all $s$.
\end{itemize}

Trivially, for \lce reals~$\alpha$ and~$\beta$, each of the two latter conditions implies the first one. For a proof of the reverse implications, assume that condition $(*)$ holds for some constant~$c$, and \lce approximations $\{a_s\}_{s \in \omega}$ and $\{b_s\}_{s \in \omega}$ of reals~$\alpha$ and~$\beta$, respectively, where we can assume~$a_0 = 0$. In case~$\alpha$ is rational, condition $(*')$ follows easily, details are left to the reader. Otherwise, given a \lce approximation $\{a'_s\}_{s \in \omega}$ of $\alpha$, for each $s$, let $f(s)$ be the largest index such that $a'_s \geq a_{f(s)}$. The function~$f$ is well-defined since~$a_0=0$, and it is a computable order since by case assumption, all values~$a_s$ and~$a_s'$ differ from~$\alpha$. Condition $(*')$ then follows by letting~$b'_s = b_{f(s)}$. Given a \lce approximation $\{b''_s\}_{s \in \omega}$ of $\beta$, for each $s$, let $g(s)$ be the smallest index such that $b''_s \le b_{g(s)}$. By assumption on~$\beta$, all values~$b_s$ and~$b_s''$ differ from~$\beta$, hence the function~$g(s)$ is a well-defined computable order, and condition $(*'')$ follows by letting~$a''_s = a_{g(s)}$.

These equivalences indicate that in a setting of \lce reals, the notion of Solovay reducibility is robust and, intuitively speaking, $\alpha \le_S \beta$ amounts to $\alpha$ being approximable by \lce approximations at least as fast as~$\beta$ up to a constant factor.

\subsection{Martin-Löf randomness}
In algorithmic randomness, various concepts of random reals or, essentially equivalent, random infinite binary sequences are investigated. The concept that is most central and receives the most attention in the literature is Martin-Löf randomness, which can be defined in several equivalent ways; see the monograph by Downey and Hirschfeldt~\cite*{rodenisbook} for details. In what follows, the term random is used as a shorthand for Martin-Löf random. 

\subsection{Convergence rate of random \lce reals}
The degree structure induced by Solovay reducibility on the class of \lce reals turned out to be rich and has been well studied~\cite*{solovaydraft, SCALUDE2001125, downey2002randomness, downey2007undecidability}.
A fundamental result is the following theorem due to Solovay~\cite{solovaydraft} and to Ku\v{c}era and Slaman~\cite{kucera2001randomness}.
\begin{theorem}[Solovay, Ku\v{c}era and Slaman]\label{thm:solocomp}
A \lce real is Solovay complete if and only if it is \ml random.
\end{theorem}
Here, a \lce real is \emph{Solovay complete} if every \lce real is Solovay reducible to it.
One way to interpret this theorem is that among \lce reals exactly the random ones converge as slowly as possible.

The most well-known example of a random \lce real is Chaitin's $\Omega$~\cite*{chaitin1975theory}, which is defined as the halting probability of a universal prefix-free Turing machine. In fact, it is known that every random \lce real is the halting probability of some universal prefix-free machine~\cite*{solovaydraft, SCALUDE2001125,downey2002randomness,downey2007undecidability}. As a consequence, all random \lce reals share many properties with $\Omega$, and accordingly such reals are called \emph{Omega numbers}.

Since computable approximations to a random \lce real converge so slowly, their speed of convergence can be set as a benchmark for the speed of convergence of other \lce reals. This phenomenon was further explored by Barmpalias and Lewis-Pye~\cite*{barmpaliasDifferencesHaltingProbabilities2017} and by Miller~\cite*{millerWorkBarmpaliasLewisPye2017}.

Fix a random \lce real $\Omega$ with a \lce approximation $\aset{\Omega_s}\siome$.
Given a \dce real $\alpha $ with a \dce approximation $\aset{\alpha_s}\siome $, let 
	\[	\partial\alpha = \partial\aset{\alpha_s} = \lim_{s\rightarrow \infty}\frac{\alpha - \alpha_s}{\Omega - \Omega_s}.	\]
For this notion of $\partial\alpha$, Barmpalias and Lewis-Pye~\cite*{barmpaliasDifferencesHaltingProbabilities2017} showed the following result.  
\begin{theorem}[Barmpalias and Lewis-Pye]\label{thm:nonranlcewd}
For any random \lce real $\alpha$ with a \lce approximation $\aset{\alpha_s}\siome $, the limit $\partial \alpha$ exists and is independent of the choice of the approximation $\aset{\alpha_s}\siome $.
\end{theorem}
Accordingly, $\partial \alpha$ is well-defined for any \lce real $\alpha$.
Miller~\cite*{millerWorkBarmpaliasLewisPye2017} observed that this result can be easily extended to all \dce reals, in the following nicer form.
\begin{theorem}[Barmpalias and Lewis-Pye, Miller]\label{thm:nonrandcezero}
For any \dce real $\alpha$,	$\partial \alpha$ is well-defined and $\partial \alpha = 0 $ if and only if $\alpha$ is not random.
\end{theorem}
Theorem~\ref{thm:nonrandcezero} implies that all the \dce approximations to a random \dce real converge at exactly the same speed.

\subsection{Definition of speedability}
Given the previous discussion of comparing the convergence rates of approximations of two different \lce reals, it is natural to consider the convergence rates of different computable approximations of a single real, and the possibility of speeding up the convergence of a given computable approximation. This leads to the notion of \emph{speedability} of a computable approximation.
\begin{definition}\label{def:rho-speedable}
Let $\aset{a_s}\siome $ be a computable approximation that converges to a real~$\alpha$ and let~$\rho$ be some rational in $[0,1)$. The approximation $\aset{a_s}\siome $ is {\em $\rho$-speedable} if either~$\alpha$ is rational or, if~$\alpha$ is not rational, if there is a computable order~$f$ such that
\begin{equation}\label{delta-2-speedability-property}
\liminf_{s\to\infty}\abs*{\frac{\alpha-a_{f(s)}}{\alpha-a_s}} \le \rho.
\end{equation}
The approximation $\aset{a_s}\siome $ is {\em speedable} if it is $\rho$-speedable for some $\rho\in [0,1)$.
\end{definition}
Note that a computable approximation cannot be $\rho$-speedable for negative~$\rho$ and is always $\rho$-speedable for~$\rho\ge 1$, hence, in connection with $\rho$-speedability, only values of~$\rho$ in the interval~$[0,1)$ are relevant. For \lce approximations, as long as~$\rho$ is nonzero, the actual choice of~$\rho$ within the latter interval does not matter by the following result of Merkle and Titov~\cite[Lemma 8]{merkleSpeedableLeftcNumbers}.
\begin{lemma}[Merkle and Titov]\label{lem:rhoindef}
If a \lce approximation is $\rho$-speedable for some $\rho < 1$, then it is $\rho$-speedable for all $\rho\in (0,1)$.
\end{lemma}
\begin{remark}
The proof in Merkle and Titov~\cite*{merkleSpeedableLeftcNumbers} only works for \lce approximations, as it uses the monotonicity of the approximation in an essential way. However, by Theorem~\ref{thm:equidce} an assertion similar to Lemma~\ref{lem:rhoindef} holds for \dce approximations: if a \dce approximation is $\rho$-speedable for some $\rho <1$, then there exists a 0-speedable \dce approximation of the same real. The latter assertion remains true when both occurrences of \dce are replaced by \can; we will, however, show the stronger result that all \ca reals have a 0-speedable \ca approximation.
\end{remark}
Every c.a.\ real has a 0-speedable \ca approximation, but in case the real is random, it cannot have a speedable \lce approximation. This suggests investigating the speedability of different types of computable approximations of the various types of c.a.\ reals. 
\begin{definition}\label{def:realspd}
Let $\alpha$ be a real, and $\rho$ be some rational.
Let $\Xs$ stand for one of the terms \lcen, \rcen, \dcen, or \can
\begin{itemize}
	\item The real~$\alpha$ is {\em $\Xs$ ($\rho$-)speedable} if it has a ($\rho$-)speedable $\Xs$ approximation.
	\item  The real~$\alpha$ is {\em strongly $\Xs$ ($\rho$-)speedable} if it has an $\Xs$ approximation that is ($\rho$-)speedable via the function $s \mapsto s+1$.
	\item  The real~$\alpha$ is {\em fully $\Xs$ speedable} if it is $\Xs$ and all of its $\Xs$ approximations are speedable.
	\item  The real~$\alpha$ is {\em weakly $\Xs$ speedable} if there exist a rational $\rho\in (0,1)$ and two $\Xs$ approximations~$\aset{a_s}\siome $ and $\aset{b_s}\siome $ of $\alpha$ such that 
	\begin{equation}
	\liminf_{s\to\infty}\abs*{\frac{\alpha-b_s}{\alpha-a_s}} \le \rho.
	\end{equation}
\end{itemize}
\end{definition}
For all $\Xs \in \aset{\text{\lcen}, \text{ \rcen}, \text{ \dcen}, \text{ \can}}$, the following hold. 
\begin{flalign}
&& \label{eq:spdimplications3} \text{$\Xs$ 0-speedability} &&   \Rightarrow  &&& \text{$\Xs$ speedability}; & \\
&&\label{eq:spdimplications1} \text{strong $\Xs$ speedability} &&   \Rightarrow  &&& \text{$\Xs$ speedability}; & \\
&&\label{eq:spdimplications2} \text{full $\Xs$ speedability} &&  \Rightarrow  &&& \text{$\Xs$ speedability}; & \\
&&\label{eq:spdimplications4} \text{$\Xs$ speedability} &&   \Rightarrow  &&& \text{weak $\Xs$ speedability}. & 
\end{flalign}
The first three implications hold by definition. For a proof of \eqref{eq:spdimplications4}, let~$\alpha$ be a real that has an $\Xs$ approximation~$\aset{a_s}\siome$ that is speedable via some computable order~$f$. Then~$\aset{a_{f(s)}}\siome$ is an $\Xs$ approximation of~$\alpha$ by Remark~\ref{remark:dyadic-approximation}, and the approximations~$\aset{a_s}\siome$ and $\aset{a_{f(s)}}\siome$ witness that~$\alpha$ is weakly $\Xs$ speedable.

\subsection{Left-c.e.\ speedability}\label{subsec:leftce-speedability}
For \lce in place of $\Xs$, the implications in \eqref{eq:spdimplications1}, \eqref{eq:spdimplications2}, and \eqref{eq:spdimplications4} can be reversed.
\begin{theorem}[Merkle and Titov]\label{thm:equilce}
For a \lce real $\gamma$, the following are equivalent.
\begin{enumerate}
\item[(i)] $\gamma$ is \lce speedable.
\item[(ii)] $\gamma$ is strongly \lce $\rho$-speedable for all $\rho\in (0,1)$.
\item[(iii)] $\gamma$ is fully \lce speedable.
\item[(iv)] $\gamma$ is weakly \lce speedable.
\end{enumerate}
\end{theorem}
\begin{remark}\label{remark:lcewspd}
The equivalence of~(i), (ii), and~(iii) was shown by~Merkle and Titov~\cite*{merkleSpeedableLeftcNumbers}. Concerning~(iv), it follows as a special case of~\eqref{eq:spdimplications4} that \lce speedability implies weak \lce speedability. For a proof of the reverse implication, we can assume that $\alpha$ is not rational. Suppose~$\alpha$ satisfies the definition of weakly \lce $\rho$-speedable via two \lce approximations~$\aset{a_s}\siome $ and $\aset{b_s}\siome $. We define a function $f\colon \Nat\to \Nat$ by letting $f(n)$ be the least number such that $a_{f(n)}\ge b_n $. It is easy to check that $f$ is a well-defined computable order such that $\alpha$ is $\rho$-speedable via $f$.
\end{remark}

Theorem \ref{thm:nonranlcewd} implies that  \lce random reals cannot be \lce speedable. Following the latter observation, Merkle and Titov~\cite*{merkleSpeedableLeftcNumbers} asked whether this yields a characterization of the random \lce reals.
\begin{question}\label{ques:nonrandspd}
Are all nonrandom \lce reals \lce speedable?
\end{question}
This question was later answered in the negative by Hölzl and Janicki~\cite{holzlBenignApproximationsNonspeedability2023}.

By Lemma~\ref{lem:rhoindef}, every speedable \lce approximation is $\rho$-speedable for all~$\rho>0$.
It is then natural to ask whether this inclusion can be extended to the case~$\rho=0$, \ien, whether \eqref{eq:spdimplications3} can be reversed for \lce in place of $\Xs$.
\begin{question}\label{ques:speedzerospd}
Are all \lce speedable reals \lce 0-speedable?
\end{question}
This question was also answered in the negative by Hölzl, Janicki, Merkle, and Stephan~\cite*{holzlRandomnessSuperspeedability2024}.

\subsection{Our results and outline of the paper}
The main goal of this paper is to study the speedability notions in Definition~\ref{def:realspd} and their relation with randomness. Note that given a \rce approximation $\aset{a_s}\siome $ of a real $\alpha$, we can consider the \lce approximation $\aset{1-a_s}\siome$ and its limit~$1-\alpha$.
Consequently, \rce speedability is fully symmetric to \lce speedability, and all results about \lce speedability can be translated to \rce speedability. In particular, Theorem~\ref{thm:equilce} holds for \rce in place of \lce as well.

In Section~\ref{sec:speedupdce}, we will show that for \dce in place of $\Xs$, the implications~\eqref{eq:spdimplications3}, \eqref{eq:spdimplications1} and~\eqref{eq:spdimplications4} can be reversed. Moreover, we will prove that a \dce real is \dce speedable if and only if it is nonrandom. Thus, Questions \ref{ques:nonrandspd} and \ref{ques:speedzerospd} both have positive answers in the \dce case.

By the negative answer to Question~\ref{ques:nonrandspd} by Hölzl and Janicki~\cite{holzlBenignApproximationsNonspeedability2023}, there is a nonrandom \lce real that is not \lce speedable. The latter real, on the one hand, it is not fully \dce speedable, since it has a \lce approximation which is not speedable and all \lce approximations are also \dce approximations; on the other hand, it is \dce speedable by our result, since it is nonrandom.
Thus, the reverse of the implication \eqref{eq:spdimplications2} does not hold for \dce in place of $\Xs$. However, as shown in Section~\ref{sec:fuddcespd}, the reverse implication holds when restricting attention to \dce reals that are neither \lce nor \rce

In Section~\ref{sec:caspd}, we study the speedability of \ca reals. We first prove that every \ca real is speedable. With a more involved argument, we then show that every \ca real is strongly \ca 0-speedable.

Though Question~\ref{ques:speedzerospd} has a negative answer, the following question is still of interest.
\begin{question}\label{quz:lceran}
Which \lce reals exactly are speedable?
\end{question}
We would like to obtain a characterization of the \lce speedable reals among the \lce reals by some randomness notion. This would lead to a characterization of the fully \dce speedable reals among the \dce reals as well. In Section~\ref{sec:speedability-and-randomness}, as a step toward answering Question~\ref{quz:lceran}, we characterize the \lce speedable reals as the reals that can be covered by suitably restricted Solovay tests.

\subsection{Notation}
For standard notation and definitions in computability and randomness, we refer to the monograph by Downey and Hirschfeldt~\cite{rodenisbook}. In the remainder of this section, we recall some notation specific to this article and adopt some conventions.

The set of natural numbers, which includes~$0$, is usually denoted by~$\Nat$, and is denoted by~$\omega$ when used in subscripts. A real with binary expansion~$0.b_0 b_1 b_2 \ldots$ is identified with the infinite binary sequence~$b_0\concat b_1\concat b_2 \cdots$. For a dyadic rational, it is always assumed that its binary expansion ends with infinitely many zeros.

An approximation $\aset{a_s}\siome $ of a real~$\alpha$ is \emph{two-sided} if there are infinitely many $s$ such that $a_s < \alpha$ and infinitely many $s$ such that $a_s > \alpha$. Such an approximation is \emph{left-sided} if for all but finitely many~$s$, $a_s \le \alpha$, and it is \emph{right-sided} if for all but finitely many~$s$, $a_s \ge \alpha$.

\begin{remark}\label{remark:one-sided}
For a real~$\alpha$ that is neither \lce nor \rcen, all its computable approximations must be two-sided. For a proof by contradiction, assume that the real has a left-sided computable approximation~$\aset{a_s}\siome$, and fix an index~$t$ such that~$a_s \le \alpha$ for all~$s \ge t$.
Letting~$b_s$ be the maximum of~$a_{t}, \ldots, a_{t+s}$ yields the \lce approximation~$\aset{b_s}\siome$ of~$\alpha$, contrary to assumption. By a symmetric argument, $\alpha$ cannot have a right-sided computable approximation, either.
\end{remark}

The term \emph{string} always refers to a finite binary string. For a string~$\sigma$, its \emph{length} is denoted by~$|\sigma|$.
Given two strings~$\sigma$ and~$\tau$, we write $\sigma \prec \tau$, in case $\sigma$ is a proper prefix of $\tau$. For a set $A$ of strings, we say $A$ is \emph{prefix-free} if no string in $A$ is a prefix of another string in $A$.

For a string~$\sigma$, let $\extclass{\sigma}$ denote the set of all infinite binary sequences that extend $\sigma$. For a set $E$ of strings, we define $\extclass{E} = \bigcup_{\sigma \in E} \extclass{\sigma}$. Let $\mu$ denote the Lebesgue measure. For simplicity, we write $\mu(\extclass{E})$ as $\mu(E)$.

We identify a \emph{\ml test} with a uniformly computable sequence~$\aset{U_i}\iiome$ of prefix-free sets of strings such that $\mu(U_i) \le 2^{-i}$ for all~$i$.
A real~$\gamma$ is \emph{random} if there is no \ml test~$\aset{U_i}\iiome$ such that~$\gamma \in \bigcap_i \extclass{U_i}$.

We also identify a \emph{Solovay test} with a sequence~$\aset{[\ell_i, r_i]}\iiome$ of nonempty intervals such that $\aset{\ell_i}\iiome$ and~$\aset{r_i}\iiome$ are computable sequences of dyadic rationals in $[0,1]$ and the sum over the terms~$r_i - \ell_i$ converges. It is known that a real~$\gamma$ is not random if and only if there is such a Solovay test such that~$\gamma \in [\ell_i, r_i]$ for infinitely many~$i$, and in this situation we say the real~$\gamma$ is covered by or fails the Solovay test.

\section{D.c.e.\ speedability}\label{sec:speedupdce}
\subsection{D.c.e.\ speedability and randomness}
\begin{theorem}\label{thm:equidce}
For a \dce real $\gamma$, the following are equivalent.
\begin{enumerate}
	\item[(i)] \label{itm:dcessp} $\gamma$ is strongly \dce 0-speedable.
	\item[(ii)] \label{itm:dcesp} $\gamma$ is \dce speedable.
	\item[(iii)] \label{itm:dcewsp} $\gamma$ is weakly \dce speedable.
	\item[(iv)] \label{itm:dceran} $\gamma$ is nonrandom.
\end{enumerate}
\end{theorem}

\begin{proof}
The chain of implications (i) $\Rightarrow$ (ii) $\Rightarrow$ (iii) follows by~\eqref{eq:spdimplications1} and~\eqref{eq:spdimplications4}, while
the implications (iii) $\Rightarrow$ (iv)  and (iv) $\Rightarrow$ (i) are precisely the statements of Lemmas~\ref*{lem:randce} and~\ref*{lem:nonrandce}, respectively.
\end{proof}

\begin{lemma}\label{lem:randce}
	If $\gamma$ is a random \dce real, then $\gamma$ is not weakly \dce speedable.
\end{lemma}
\begin{proof}
Let $\gamma$ be a random \dce real, and let $\aset{a_s}\siome $ and $\aset{b_s}\siome $ be any two \dce approximations of $\gamma$.
By Theorem \ref{thm:nonrandcezero}, we have $\partial\aset{a_s} = \partial\aset{b_s} = \partial \gamma \neq 0$, i.e.,
\[
\lim_{s\rightarrow \infty}\frac{\gamma - a_s}{\Omega - \Omega_s} = \lim_{s\rightarrow \infty}\frac{\gamma - b_s}{\Omega - \Omega_s} \neq 0.
\]
Then 
\[	
\lim_{s\rightarrow \infty}\frac{\gamma - b_s}{\gamma - a_s } = \lim_{s\rightarrow \infty}\frac{ \frac{\gamma - b_s}{\Omega - \Omega_s} } { \frac{\gamma - a_s}{\Omega - \Omega_s} } = 1.
\]
Thus, $\gamma$ is not weakly \dce speedable.
\end{proof}

Merkle and Titov~\cite{merkleSpeedableLeftcNumbers} observed that it is immediate from Theorem \ref{thm:nonranlcewd} that a random \lce real cannot be \lce speedable. Moreover, they gave a direct combinatorial proof for the latter assertion. Similar to the \lce case, here we also give a direct combinatorial proof of Lemma~\ref{lem:randce}.

\begin{proof}[Alternative proof of Lemma~\ref*{lem:randce}]
For a proof by contradiction, assume that for some random real~$\alpha$, there are \dce approximations $\aset{a_s}\siome$ and~$\aset{b_s}\siome$ of $\alpha$ and $\rho\in(0,1)$ such that
\begin{equation}\label{eq:weakspdpr}
\abs*{\alpha-b_s} \le \rho \abs*{\alpha-a_s} \text{ for infinitely many } s.
\end{equation}

By an argument of Rettinger~\cite[Theorem~9.2.4]{rodenisbook}, \dce approximations of the random real~$\alpha$ cannot be two-sided. Thus, the two approximations~$\aset{a_s}\siome$ and~$\aset{b_s}\siome$ must be left-sided or right-sided. Actually, they must be either both left-sided or both right-sided; otherwise, the real~$\alpha$ would be \lce and \rcen, hence computable.
Without loss of generality, we can assume that both approximations are left-sided.

We define a function~$f\colon \Nat\to \Nat$ by letting~$f(0)=1$ and for all~$s\ge0$,
\[	f(s+1) = \min \aset{ i > f(s) \colon a_i \ge b_{s+1} }. 	\]
It is easy to check that $f$ is well-defined and is a strictly increasing computable function where $f(s) \ge s+1$ and $a_{f(s)} \ge b_s$ for all $s$. Moreover, by \eqref{eq:weakspdpr}, we have
\begin{equation}\label{eq:weakspdpr2}
\alpha-a_{f(s)} \le \rho (\alpha-a_s)  \text{ for infinitely many } s.
\end{equation}
We partition the natural numbers into consecutive intervals~$\aset{J_r}\riome$, which are defined inductively. Let~$J_0=\aset{0}$, and for $r\ge0$ let
\[
J_{r+1} = \{\max J_r + 1, \max J_r + 2, \ldots,  f(\max J_r) \}.
\]
By construction and since the function~$f$ is strictly increasing, for all~$r$ and all~$s$ in~$J_r$, the value~$f(s)$ is in~$J_r \cup J_{r+1}$. For all~$r$ let
\[
c_r = \min \{ a_s  \colon s \in J_r  \} \makebox[5em]{ and } 
d_r = \max \{ a_s  \colon s \in J_r \cup J_{r+1} \} .
\]
Now fix indices~$s$ and~$r$ such that~$s$ is in~$J_r$, and note that $c_r \le a_s$ and $d_r \ge a_{f(s)}$. Moreover, $a_j < \alpha$ holds for all~$j$.
Since every interval~$J_r$ is finite, by~\eqref{eq:weakspdpr2}, for infinitely many~$r$ we have
\begin{equation}\label{eq:weakspdpr3}
\alpha- d_r \le \rho (\alpha- c_r), \text{ i.e., } \alpha \le c_r + \frac1{1-\rho} (d_r - c_r).
\end{equation}
Let $I_r = [c_r,  c_r + \frac1{1-\rho} (d_r - c_r)]$. Then $\alpha \in I_r$ for infinitely many~$r$.

On the other hand, as $d_r \in \aset{a_s \colon s \in J_r \cup J_{r+1}}$ and $c_r \in \aset{a_s \colon s \in J_r}$, it follows by the triangle inequality that $d_r - c_r \le \sum_{s \in J_r \cup J_{r+1}} \abs{a_{s+1} - a_s}$.
Then
\[
\sum_{r \in \omega} (d_r - c_r) \le \sum_{r \in \omega} \sum_{s \in J_r \cup J_{r+1}} \abs{a_{s+1} - a_s} \le 2 \sum_{s \in \omega} \abs{a_{s+1} - a_s} < \infty.
\]
Therefore, the sum of the lengths of the intervals~$I_r$ is finite.
Thus~$\aset{I_r}\riome$ is a Solovay test that covers the random real~$\alpha$, a contradiction.
\end{proof}

\begin{lemma}\label{lem:nonrandce}
	If $\gamma$ is a nonrandom \dce real, then $\gamma$ is strongly \dce $0$-speedable.
\end{lemma}

\begin{proof}
We can assume that~$\gamma$ is not rational because, otherwise, $\aset{\gamma - 2^{-s^2}}\siome$ is a strongly \dce $0$-speedable approximation to~$\gamma$. Fix some \dce approximation~$\aset{r_j}\jiome $ of~$\gamma$. By Remark~\ref{remark:dyadic-approximation}, we can assume that the~$r_j$ are pairwise distinct dyadic rationals. Let~$\tau_j$ be the unique string with no trailing zeros such that~$r_j=0.\tau_j$. As $\gamma $ is not random, let $\aset{U_i}\iiome $ be a \ml test which covers $\gamma$, i.e., we have for all~$i$
\[	 \mu(U_i) \le 2^{-i} \makebox[4em]{ and }  \gamma \in \extclass{U_i} .\]
As $\aset{U_i}\iiome $ are uniformly c.e., we can assume that the sets~$U_i$ are pairwise disjoint.
This can be achieved by modifying the uniform enumeration process as follows: whenever there is some string $\sigma$ to be enumerated into $U_j$ that has already appeared before in some other set $U_i$, instead of $\sigma$, enumerating into $U_j$ all extensions of $\sigma$ of a suitable length such that these strings have not appeared as far in the whole enumeration.
Let~$\sigma_0, \sigma_1, \ldots$ be an effective enumeration of all the strings in $U_{2k}$ for all $k$.

We construct an approximation $\aset{c_s}\siome $ of $\gamma$ such that 
\begin{equation}\label{eq:nonrandce1}
\liminf_{s\rightarrow \infty}\abs*{\frac{\gamma - c_{2s+1}}{\gamma - c_{2s} }} = 0.
\end{equation}

\medskip \noindent {\bf Construction: }
\textit{At stage~$0$:} declare all strings~$\sigma_i$ except~$\sigma_0$ to be unused, and let~$i(0)=j(0)=0$.\\
\textit{At stage~$s>0$:} let~$j(s)$ be the minimum index~$j>j(s-1)$ such that for some~$i \le j$, the string~$\sigma_i$ is yet unused and a prefix of~$\tau_j$. (We show the existence of such indices~$i$ and~$j$ shortly in the verification part below.)
Among all such indices~$i$, let~$i(s)$ be equal to the minimum one and declare~$\sigma_{i(s)}$ to be used. 
We let~$k(s)$ be the unique index such that~$\sigma_{i(s)}$ is in ~$U_{2 k(s)}$. 
Then we let
\[
c_{2s} = r_{j(s)} + \delta_s \makebox[3.5em]{ and  }	
c_{2s+1} = r_{j(s)} 
\qquad\makebox[3.5em]{ where  } \delta_s = 2^{-|\sigma_{i(s)}|+k(s)}.
\]

\medskip \noindent {\bf Verification: }
We call a string \emph{true}, if it is in $\aset{\sigma_i}\iiome$ and is a prefix of~$\gamma$.
By choice of the~$U_i$, each~$U_{2k}$ contains a true string.
Then there are infinitely many true strings.
As~$\aset{r_t}\tiome$ converge to~$\gamma$, thus each of the infinitely many true strings~$\sigma_r$ is a prefix of~$\tau_t$ for almost all~$t$. So in each stage there are indices~$j$ and~$i$ as required. Consequently, the sequences~$\aset{i(s)}\siome$ and~$\aset{{j(s)}}\siome$ are well-defined. Furthermore, both sequences are computable because the construction is effective. 
By definition, the sequence~$\aset{j(s)}\siome$ is strictly increasing.
Since we never make a string~$\sigma_i$ unused again after it has been used, every string $\sigma_i$ is used at most once.

For every true string~$\sigma_r$, there must be a stage~$s$ such that $\sigma_r$ is a prefix of $\tau_j$ for all $j \ge s$.
Then if $\sigma_r$ has not been used at stage~$s$ yet, by definition, $\sigma_r$ must be used in at most~$r$ many more stages. Thus, every true string is used eventually.

We call a stage \emph{true}, if a true string is used at this stage.
Let $s$ be a true stage. The strings~$\tau_{j(s)}$ and~$\gamma$ share the prefix~$\sigma_{i(s)}$, hence the reals~$r_{j(s)}= c_{2s+1}$ and~$\gamma$ differ at most by~$2^{-|\sigma_{i(s)}|}$.
On the other hand, the values~$c_{2s}$ and~$r_{j(s)}$ differ by~$\delta_s = 2^{-|\sigma_{i(s)}|+k(s)}$, hence~$c_{2s}$ and~$\gamma$ differ by at least~$(2^{k(s)}-1) 2^{-|\sigma_{i(s)}|}$.
Thus,
\begin{equation}\label{eq:nonrandce2}
\abs*{\frac{\gamma - c_{2s+1}}{\gamma - c_{2s} }}	\le \frac{1}{2^{k(s)}-1} \text{ for every true stage $s$.}
\end{equation}
Since each $U_{2k}$ is prefix-free, it contains only one true string. 
On the other hand, in every stage the string used is distinct.
Then $k(s)$ are different for different true stages~$s$.
Thus, \eqref{eq:nonrandce2} implies \eqref{eq:nonrandce1}, as there are infinitely many true stages.

To complete the proof, it remains to show that the sequence~$\aset{c_i}\iiome$ is a \dce approximation of~$\gamma$. 
It suffices to show that  the sum over the terms~$|c_{i+1}-c_{i}|$ is finite. Note that the latter implies, in particular, that the sequence~$\aset{c_i}\siome$ converges, hence converges to the same limit~$\gamma$ as its subsequence~$\aset{r_{j(s)}}\siome$. By definition, we have
\[
|c_{2s+1}-c_{2s}| \le \delta_s
\makebox[4em]{ and }	 
|c_{2s+2}-c_{2s+1}|	 
\le | r_{j(s+1)}  - r_{j(s)}| + \delta_{s+1}.
\]
By the triangle inequality and because~$\aset{r_{j(s)}}\siome$ is an increasing subsequence of the \dce approximation~$\aset{r_j}\jiome$, the sum over the terms~$|r_{j(s+1)}  - r_{j(s)}|$ is finite.
The sum over the~$\delta_s$ is also finite since each~$\sigma_i$ is used at most once. We have 
\[
\sum_{s \in \omega } \delta_s 
= 	\sum_{s \in \omega } 2^{-|\sigma_{i(s)}|+k(s)}
\le \sum_{k\in \omega } \sum_{\sigma \in U_{2k} } 2^{-|\sigma|+k}
\le \sum_{k\in \omega } 2^{k} \sum_{\sigma \in U_{2k}} 2^{-|\sigma|}
\le \sum_{k\in \omega } 2^{-k} = 2.
\]

\end{proof}

\subsection{Full \dce speedability}\label{sec:fuddcespd}

\begin{lemma}\label{lem:twsspee} 
Every two-sided computable approximation is speedable.
\end{lemma}

For the proof, we need the following technical fact.
\begin{proposition}\label{prop:nocover}
Given an interval $[a,b]$ and an array of intervals $I_0 = [z_0-d_0,z_0+d_0]$, $\dots$, $I_m = [z_m-d_m,z_m+d_m]$, if $d_i \ge (b-a)/m$ for every $0\le i \le m $ and $z_j\notin I_i$ for all $0\le i < j \le m $, then there exists $i\in\aset{0,\dots,m}$ such that $z_i \notin [a,b]$.
\end{proposition}
\begin{proof}
For all $0\le i < j \le m $, as $z_j\notin I_i$, we have $\abs{z_j-z_i} > d_i \ge (b-a)/m$.
Rearrange $z_0,\dots,z_m$ in order of magnitude such that $z_{\pi(0)}<\dots<z_{\pi(m)}$.
Then $z_{\pi(i)} - z_{\pi(i-1)} > (b-a)/m$ for all $0< i \le m$.
Thus, 
\[	z_{\pi(m)} - z_{\pi(0)} = \sum_{i=1}^m (z_{\pi(i)} - z_{\pi(i-1)} ) > m\cdot \frac{b-a}{m} = b-a.	\]
Therefore, the interval $[a,b]$ cannot cover both $z_{\pi(m)}$ and $z_{\pi(0)}$.
\end{proof}

\begin{proof}[Proof of Lemma~\ref*{lem:twsspee}]
Given $\aset{a_s}\siome $ as in the hypothesis and $\rho \in (0,1)$, for a contradiction suppose that $\aset{a_s}\siome$ is not $\rho$-speedable.

Let $c = \rho/(1+\rho)$ and $k $ be an integer such that $ k \ge 1+ c^{-2} $.
For every $i\in\aset{0,\cdots,2k}$, we inductively define a computable order $f_i$ and a constant $n_i$ such that
\begin{equation}\label{eq:require}
	n_i \ge n_{i-1} \text{ and } \alpha\notin I_i(n) \text{ for all } n\geq n_i,
\end{equation}
where 
\[ n_{-1} = 0 \text{ and } I_i(n) = \big[ a_{f_i(n)} - c \abs{a_{f_i(n)} - a_n},\, a_{f_i(n)} + c \abs{a_{f_i(n)} - a_n} \big]. \]
First, let 
\[ f_0(n) = n+1. \]
As $\aset{a_s}\siome$ is not $\rho$-speedable, there exists an index~$s$ such that, for every~$n\geq s$, it holds that~$\abs*{\frac{\alpha - a_{f_0(n)}}{\alpha - a_n}}>\rho$. It implies that $\alpha\notin I_0(n)$.
We define $n_0$ to be the least such index $s$.
Thus, \eqref{eq:require} is satisfied for $i=0$.

Given $1\le i\le 2k$, suppose $f_j$ and $n_j$ are defined and satisfy \eqref{eq:require} for all $j\in\aset{0,\cdots,i -1 }$.
Then for all $n\ge n_{i-1}$ we have $\alpha \notin \bigcup_{j=0}^{i - 1} \bigcup_{\ell=n_j}^{n}I_j(\ell)$.
Thus, 
\begin{equation}\label{eq:betcoeff}
	f_i(n) = \begin{cases}
	n+1	& 	\mbox{ if }	 n < n_{i-1}, \\
	\min\aset{m>n: a_m\notin \bigcup_{j=0}^{i-1} \bigcup_{\ell=n_j}^{n}I_j(\ell)}	& \mbox{ otherwise.}
	\end{cases}		
\end{equation}
is well-defined.
It is easy to check that $f_i$ is a strictly increasing function. Now that $f_i$ is not a $\rho$-speedup function for $\aset{a_s}\siome$, there exists an index~$s \ge n_{i-1}$ such that, for every~$n\geq s$, it holds that~$\abs*{\frac{\alpha - a_{f_i(n)}}{\alpha - a_n}}>\rho$, which implies that~$\alpha\notin I_i(n)$.
We define $n_i$ to be the least such index $s$.
Then \eqref{eq:require} is satisfied for $i$.
This completes the inductive definition of $f_i$ and $n_i$ for all $i\in\aset{0,\cdots,2k}$.

As $\aset{a_s}\siome $ is a two-sided approximation to $\alpha$, let $s\ge n_{2k}$ be an index such that $a_s<\alpha<a_{s+1}$.

Now we derive a contradiction as promised. The proof idea is as follows.
We first find the index $u\ge s+1$ such that $a_u$ is the last one to the right of or equal to $a_{s+1}$.
Then all of $a_{f_1(u)}$, $a_{f_2(u)}$, $\dots$, $a_{f_k(u)}$ are to the left of $a_{s+1}$. Moreover, they must jump over the interval $[a_{s+1}-c(a_{s+1} - a_s), a_{s+1}]$. Then their distances from $a_u$ are bounded below. Thus, by our arrangement of the parameters, we argue that at least one of them must be to the left of $a_s$. This is illustrated in Figure~\ref{fig:jumpover}.
\begin{figure}[!t]
\centering
\begin{tikzpicture}
\tikzmath{\afu = 2; \as = 4; \afua = 5.5; \aa = 7; \afuk = 8.5; \afub = 10; \afs = 12; \au = 14; 
\cc=0.12; \cafs=\cc*(\afs-\as); \cafua=\cc*(\au-\afua); \cafub=\cc*(\au-\afub); \cafuk=\cc*(\au-\afuk); 
}

\draw[->,-{Stealth[length=8pt]}] (0,0) -- (14.5,0);
\fill (\aa,0) circle (2pt) node[below=2pt]  {$\alpha$};

\fill (\as,0) circle (2pt) node[below=2pt]  {$a_s$};
\fill (\afs,0) circle (2pt) node[below=2pt]  {$a_{s+1}$};
\draw[->,-{Stealth[length=8pt]}] (\as,0) to[out=55, in=125] node[pos=0.84, right] {$f_0$} (\afs,0);
\draw[line width=1.5pt, arrows = {Bracket-Bracket}] (\afs-\cafs,0) -- (\afs+\cafs,0) node[above=2pt] {$I_0(s)$};
\draw [|<->|] (\afs-\cafs,-0.7) -- node[below=2pt] {$c(a_{s+1} - a_s)$} (\afs,-0.7);

\fill (\au,0) circle (2pt) node[below=2pt] {$a_u$};
\fill (\afua,0) circle (2pt) node[below=2pt] {$a_{f_1(u)}$};
\draw[dashed,->,-{Stealth[length=8pt]}] (\au,0) .. controls (\afs+1,3) and (\aa+1,5) .. node[pos=0.92, left] {$f_1$} (\afua,0);
\draw[line width=1.5pt, arrows = {Bracket-Bracket}] (\afua-\cafua,0) -- (\afua+\cafua,0) node[above=2pt] {$I_1(u)$};	
\fill (\afub,0) circle (2pt) node[below=2pt] {$a_{f_2(u)}$};
\draw[dashed,->,-{Stealth[length=8pt]}] (\au,0) .. controls (\afs+1,3) and (\aa+1,5) .. node[pos=0.92, left] {$f_2$} (\afub,0);
\draw[line width=1.5pt, arrows = {Bracket-Bracket}] (\afub-\cafub,0) -- (\afub+\cafub,0) node[above=2pt] {$I_2(u)$};	
\fill (\afuk,0) circle (2pt) node[below=2pt] {$a_{f_k(u)}$};
\draw[dashed,->,-{Stealth[length=8pt]}] (\au,0) .. controls (\afs+1,3) and (\aa+1,5) .. node[pos=0.92, left] {$f_k$} (\afuk,0);	
\draw[line width=1.5pt, arrows = {Bracket-Bracket}] (\afuk-\cafuk,0) -- (\afuk+\cafuk,0) node[above=2pt] {$I_k(u)$};	
\fill (\afu,0) circle (2pt) node[below=2pt] {$a_{f_i(u)}$};	
\draw[dashed,->,-{Stealth[length=8pt]}] (\au,0) .. controls (\afs+1,3) and (\aa+1,5) .. node[pos=0.92, left] {$f_i$} (\afu,0);
\end{tikzpicture}
\caption{$a_{f_1(u)}$, $a_{f_2(u)}$, $\dots$, $a_{f_k(u)}$ jump over the interval $[a_{s+1}-c(a_{s+1} - a_s), a_{s+1}]$, thus there exists $i\in\aset{1,\cdots,k}$ such that $a_{f_i(u)} < a_s$.}
\label{fig:jumpover}
\end{figure}
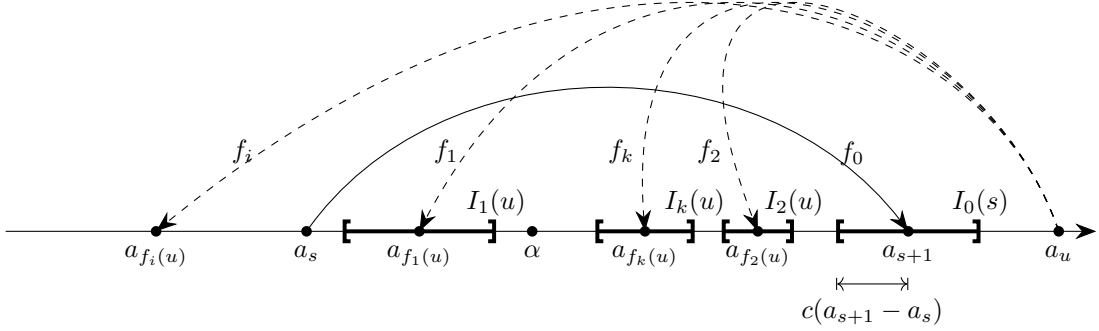
Fixing one such point $a_{f_i(u)}$, we find the index $v\ge f_i(u)$ such that $a_v$ is the last one to the left of or equal to $a_{f_i(u)}$. In a symmetrical way as before, we find one point $a_{f_{i+j}(v)}$ to be to the right of $a_u$. However, this contradicts the fact that $a_u$ is the last one to the right of or equal to $a_{s+1}$. The whole argument is illustrated in Figure~\ref{fig:backandforth}.
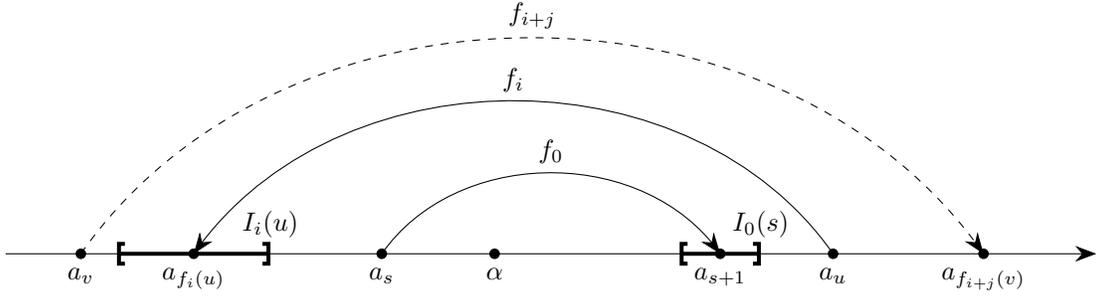
\begin{figure}[!t]
\centering
\begin{tikzpicture}
\tikzmath{\av = 1; \afu = 2.5; \as = 5; \aa = 6.5; \afs = 9.5; \au = 11; \afv = 13;
\cc=0.12; \cafs=\cc*(\afs-\as); \cafu=\cc*(\au-\afu); 
}
\draw[->,-{Stealth[length=8pt]}] (0,0) -- (14.5,0) ;
\fill (\aa,0) circle (2pt) node[below=2pt] {$\alpha$};

\fill (\as,0) circle (2pt) node[below=2pt]  {$a_s$};
\fill (\afs,0) circle (2pt) node[below=2pt]  {$a_{s+1}$};
\draw[->,-{Stealth[length=8pt]}] (\as,0) to[out=55, in=125] node[midway, above] {$f_0$} (\afs,0);
\draw[line width=1.5pt, arrows = {Bracket-Bracket}] (\afs-\cafs,0) -- (\afs+\cafs,0) node[above=2pt] {$I_0(s)$};

\fill (\au,0) circle (2pt) node[below=2pt] {$a_u$};
\fill (\afu,0) circle (2pt) node[below=2pt] {$a_{f_i(u)}$};	
\draw[->,-{Stealth[length=8pt]}] (\au,0) to[out=125, in=55] node[midway, above] {$f_i$} (\afu,0);
\draw[line width=1.5pt, arrows = {Bracket-Bracket}] (\afu-\cafu,0) -- (\afu+\cafu,0) node[above=2pt] {$I_i(u)$};

\fill (\av,0) circle (2pt) node[below=2pt] {$a_v$};
\fill (\afv,0) circle (2pt) node[below=2pt] {$a_{f_{i+j}(v)}$};
\draw[dashed,->,-{Stealth[length=8pt]}] (\av,0) to[out=55, in=125] node[midway, above] {$f_{i+j}$} (\afv,0);
\end{tikzpicture}
\caption{We first find $a_u$ to be the last one to the right of or equal to $a_{s+1}$. Then find $a_{f_i(u)}$ to the left of $a_s$, and then $a_v$ the last one to the left of or equal to $a_{f_i(u)}$, and then $a_{f_{i+j}(v)}$ to the right of $a_u$. This contradicts the fact that $a_u$ is the last one to the right of or equal to $a_{s+1}$.}
\label{fig:backandforth}
\end{figure}

Formally, let $u\ge s+1$ be the index such that
\begin{equation}\label{eq:upbound}
	a_u \ge a_{s+1} \text{ and } \forall t > u (a_t < a_{s+1}).
\end{equation}
The existence of such $u$ is guaranteed by the fact that $\aset{a_s}\siome $ converges to $\alpha$ and $\alpha<a_{s+1}$.
For each $i\in\aset{1,\cdots,k}$, by definition of $f_i$ and $s \ge n_{2k}\ge n_0$ we have
\[	a_{f_i(u)} \notin I_0(s) = \big[ a_{s+1} - c(a_{s+1} - a_s), a_{s+1} + c(a_{s+1} - a_s) \big].	\]
On the other hand, as $f_i(u)> u$,  by \eqref{eq:upbound} we have $a_{f_i(u)} < a_{s+1}$.
Thus, $a_{f_i(u)} < a_{s+1} - c(a_{s+1} - a_s)$.
Then 
\[ c\abs {a_{f_i(u)} - a_u} \ge c (a_{s+1} - a_{f_i(u)} ) > c^2 (a_{s+1} - a_s) \ge (a_{s+1} - a_s)/(k-1). \]
Moreover, by definition, $a_{f_j(u)}\notin I_i(u)$ for all $0\le i < j \le k $.
Thus, the interval $[a_s, a_{s+1}]$ and the array of intervals $I_1(u)$, $I_2(u)$, $\dots$, $I_k(u)$ satisfy the condition of Proposition~\ref*{prop:nocover}.
Therefore there exists $i\in\aset{1,\cdots,k}$ such that $a_{f_i(u)} \notin [a_s, a_{s+1}]$, \ie, $a_{f_i(u)} < a_s$. 

Fix one such $i$. Similar as before, let $v\ge f_i(u) $ be the index such that
\begin{equation}\label{eq:lowbound}
	a_v \le a_{f_i(u)} \text{ and } \forall t > v (a_t > a_{f_i(u)}).
\end{equation}
The existence of such $v$ is guaranteed by the fact that $\aset{a_s}\siome $ converges to $\alpha$ and $\alpha>a_s>a_{f_i(u)} $.
For each $j\in\aset{1,\cdots,k}$, by definition and $u \ge s \ge n_{2k}\ge n_i$ we have 
\[	a_{f_{i+j}(v)} \notin I_i(u) = \big[ a_{f_i(u)} - c(a_u - a_{f_i(u)} ),\, a_{f_i(u)} + c(a_u - a_{f_i(u)})\big].	\]
On the other hand, as $f_{i+j}(v) > v$, by \eqref{eq:lowbound}, $a_{f_{i+j}(v)} > a_{f_i(u)}$.
Thus, $a_{f_{i+j}(v)} > a_{f_i(u)} + c(a_{f_i(u)} - a_u)$.
Then 
\[ c \abs {a_{f_{i+j}(v)} - a_v} \ge c(a_{f_{i+j}(v)} - a_{f_i(u)}) > c^2 (a_{f_i(u)} - a_u)\ge (a_{f_i(u)} - a_u)/(k-1). \]
Moreover, by definition, $a_{f_{i+j}(v)}\notin I_{i+\ell}(v)$ for all $0\le \ell < j \le k $.
Thus, the interval $[a_{f_i(u)}, a_u]$ and the array of intervals $I_{i+1}(v)$, $I_{i+2}(v)$, $\dots$, $I_{i+k}(v)$ satisfy the condition of Proposition~\ref*{prop:nocover}.
Therefore, there exists $j\in\aset{1,\cdots,k}$ such that $a_{f_{i+j}(v)} \notin [a_{f_i(u)}, a_u]$. 
However, by \eqref{eq:upbound} and \eqref{eq:lowbound}, for all $t > v$, $a_t \in (a_{f_i(u)}, a_{s+1}) \subset [a_{f_i(u)}, a_u] $. As $f_{i+j}(v) > v$, this is a contradiction.
\end{proof}

The following theorem is immediate by Remark~\ref{remark:one-sided} and Lemma~\ref{lem:twsspee}.
\begin{corollary}\label{thm:dceflsp}
If a \dce real is neither \lce nor \rcen, then it is fully \dce speedable.
\end{corollary}

\section{C.a.\ speedability}\label{sec:caspd}
Similar to the \dce case, by Remark~\ref{remark:one-sided} and Lemma~\ref{lem:twsspee}, we also have the following result about fully \ca speedability.
\begin{theorem}\label{thm:caflsp}
If a \ca real is neither \lce nor \rcen, then it is fully \ca speedable.
\end{theorem}
As for \ca speedability, with the following proposition, Lemma~\ref{lem:twsspee} actually implies that all \ca reals are \ca speedable.
\begin{proposition}\label{prop:equidelta}
Every \ca real has a two-sided computable approximation.
\end{proposition}
\begin{proof}
This is obvious in case~$\alpha$ is rational, so from now on, we can assume otherwise. Fix some computable approximation $\aset{a_s}\siome$ of~$\alpha$ where, by Remark~\ref{remark:dyadic-approximation}, we can assume that the~$a_s$ are pairwise distinct dyadic rationals in the interval~$(0,1)$. For each $s > 0$, let $h(s)$ be the position of the most significant bit at which~$a_s$ and~$a_{s-1}$ differ, where, as usual, more significant means closer to the decimal point and counting of positions starts with~$1$. Let 
\[a^-_s = \max\{a_s - 2^{-h(s)}, 0\} \makebox[5em]{  and }
a^+_s = \min\{a_s + 2^{-h(s)}, 1\}.
\]
As $\aset{a_s}\siome$ converges to the irrational real~$\alpha$ and the~$a_s$ are pairwise distinct, the values~$h(s)$ tend to infinity.
Then there are infinitely many stages~$s$ such that $h(t) > h(s)$ for all $t > s$. To see this, assume otherwise, then there is $s_0$ such that for all $s \ge s_0$, there exists $t>s$ with $h(t) \le h(s)$. Then $\liminf_{s\to\infty} h(s) \le h(s_0)$, contradicts to the fact that $\lim_{s\to\infty} h(s) = \infty$.
Now for each such stage~$s$, $a_s$ and $\alpha$ agree on the bits up to position~$h(s)$. Clearly, $a^-_s < \alpha < a^+_s$.
In summary, the sequence $\aset{a^-_1, a^+_1, a^-_2, a^+_2, \dots}$ is a two-sided computable approximation of $\alpha$.
\end{proof}

\begin{corollary}\label{thm:equidelta}
Every \ca real is \ca speedable.
\end{corollary}

With a more involved argument, we show the following stronger result.
\begin{theorem}\label{theorem:ca-reals-are-zero-speedable}
Every \ca real is strongly \ca 0-speedable.
\end{theorem}

\begin{proof}
Let~$\alpha$ be a \ca real. Fix some computable approximation~$\aset{a_s}\siome$ of~$\alpha$.
By Remark~\ref{remark:dyadic-approximation}, we can assume that the~$a_s$ are pairwise distinct dyadic rationals.
Observe that since the values~$a_i$ converge to~$\alpha$, for every pair of fixed indices~$s>0$ and~$g(s)>s$ it holds that
\begin{equation}\label{eq:proof-ca-reals-are-zero-speedable}
\abs*{\frac{\alpha - a_{g(s)}}{\alpha-a_{s}}} < \frac{1}{s} 
\quad\; \text{ if and only if } \quad
\abs*{\frac{a_i - a_{g(s)}}{a_i-a_{s}}} < \frac{1}{s}
\;\;\text{ for all sufficiently large~$i$}.    
\end{equation}
We construct in stages a computable approximation~$\aset{b_s}\siome$ of~$\alpha$ that is \ca $0$-speedable via the function~$s \mapsto s+1$.
At each stage~$i>0$, we maintain a function~$g \colon \Nat \rightarrow \Nat$ with~$g(s) \ge s$ for all~$s$.
The construction resembles a finite-injury priority argument. 
We say that index~$s<i$ \emph{requires attention} at stage~$i$ in case~$s$ and the current value of~$g(s)$ do not satisfy the inequality on the right-hand side of \eqref{eq:proof-ca-reals-are-zero-speedable}.

\medskip \noindent {\bf Construction: }
\textit{At stage~$0$:} we initialize the function $g$ by letting~$g(s)=s$ for all~$s$.\\
\textit{At stage~$i>0$:} let~$s\le i$ be the minimum such that~$s$ requires attention. We increment the value~$g(s)$ by one, and let~$b_{2i}=a_s$ and~$b_{2i+1}=a_{g(s)}$.
We say the index~$s$ \emph{receives attention} at this stage.

\medskip \noindent {\bf Verification: }
Note that at stage~$i>0$ an index~$s<i$ that requires attention always exists: at the beginning of stage~$i$ there is some~$0<s\le i$ where~$g(s)=s$ since in the previous stages at most~$i-1$ values of~$g$ were incremented.

Now we show that each index~$s$ receives attention at most finitely often. Indeed, since the~$a_i$ converge to~$\alpha$, for every~$s$ there is a natural number~$m_s > s$ such that for $\delta = \abs*{\alpha - a_s} > 0$, and for all~$i, t \ge m_s$,
\[ \abs{a_i - a_s} > \frac{\delta}{2} \quad  \& \quad \abs{a_i - a_t} < \frac{\delta}{2s}, \quad \text{ which implies } \quad \abs*{\frac{a_i - a_t}{a_i-a_s}} < \frac{1}{s}. \]
Thus~$s$ will receive attention at most~$m_s-s$ times, after which~$g(s) \ge m_s$ holds.

Fix some index~$t$, and let~$i$ be a stage where~$b_{2i}=a_t$ or~$b_{2i+1}=a_t$. Then either~$t$ or an index~$s<t$ where~$t=g(s)$ receives attention at stage~$i$, which can happen at most finitely often. Thus, each value~$a_t$ occurs at most finitely often in the sequence~$\aset{b_s}\siome$.
Then the sequence~$\aset{b_s}\siome$ converges to $\alpha$ and is a computable approximation of~$\alpha$ by construction.

Finally, we show that the approximation~$\aset{b_s}\siome$ is \ca $0$-speedable via the function~$s \mapsto s+1$.
For each $s$, let $t_s$ be the stage after which index~$s$ never receives attention again. Then $t_s \ge s$, and $s$ and the final value of~$g(s)$ do satisfy the inequality on the right-hand side of \eqref{eq:proof-ca-reals-are-zero-speedable} for all~$i > t_s$, which then implies the inequality on the left-hand side of \eqref{eq:proof-ca-reals-are-zero-speedable}.
In other words, for each~$s$, there exists~$t_s \ge s$ such that
\[ \abs*{\frac{\alpha - b_{2t_s+1}}{\alpha-b_{2t_s}}} < \frac{1}{s}, \quad \text{ which then implies } \quad \liminf_{s\rightarrow\infty}\abs*{\frac{\alpha - b_{s+1}}{\alpha-b_{s}}} = 0.  \]
\end{proof}
\section{Left-c.e.\ speedability and randomness}\label{sec:speedability-and-randomness}
As discussed in Section~\ref{subsec:leftce-speedability}, the \lce speedable reals form a proper subclass of the nonrandom \lce reals~\cite{merkleSpeedableLeftcNumbers, holzlBenignApproximationsNonspeedability2023}. As a step towards a characterization of the \lce speedable reals via notions of randomness, we demonstrate in this section that a real is \lce speedable if and only if it is covered by a restricted type of Solovay test. Before that, we first discuss the connections between effective approximations and Solovay tests in general. This relates to the observation of Rettinger~\cite[Theorem~9.2.4]{rodenisbook} that given a two-sided \dce approximation~$\aset{a_s}\siome$ of some real, this real is covered by the Solovay test that is formed by the intervals~$[\min\{a_s, a_{s+1}\}, \max\{a_s, a_{s+1}\} ]$.
\begin{definition}
Let~$\aset{I_i}\iiome$ be a Solovay test where~$I_i = [\ell_i, r_i]$. Note that by our convention, $0 \le \ell_i \le r_i \le 1$ for all $i$. We refer to~$\ell_0, r_0, \ell_1, r_1, \ldots$ as the \emph{sequence of endpoints} of this Solovay test, and its \emph{sequence of left endpoints} and \emph{sequence of right endpoints} are defined accordingly. The Solovay test then   
\begin{itemize}
\item[-] is \emph{converging} if its sequence of endpoints converges,
\item[-] is \emph{\dcen} if its sequence of endpoints has bounded variation,
\item[-] is \emph{\lcen} if its sequence of left endpoints is nondecreasing,
\item[-] \emph{has bounded increments} if for some rational constant~$d>0$ it holds for all~$i$ that
    \begin{equation}\label{eq:AT}
      \ell_{i+1} \ge \ell_i + d(r_i - \ell_i) .
    \end{equation}
\end{itemize}
\end{definition}
By definition, every Solovay test with bounded increments is \lcen, and similarly, \lce implies \dcen, which in turn implies converging. Observe that a Solovay test with sequence of endpoints~$\ell_0, r_0, \ell_1, r_1, \ldots$ is \dce if and only if the sum over the terms~$\abs{\ell_{i+1}-r_i}$ is finite.

By \emph{limit} of a converging Solovay test, we refer to the limit of its sequence of endpoints.  
\begin{remark}\label{remark:bounded-solovay-covers-uniquely}
A converging Solovay test cannot cover any real other than its limit. For a proof, assume that some real~$\gamma$ is covered by a converging Solovay test~$\aset{I_i}\iiome$. Then there are infinitely many intervals~$I_i$ that contain~$\gamma$, and since the interval lengths tend to~$0$, the left endpoints of these intervals converge to~$\gamma$. So the sequence of endpoints converges to~$\gamma$ since it converges and has a subsequence that converges to~$\gamma$.
Note that a converging Solovay test may fail to cover its limit, and then does not cover any real at all. For example, for a strictly increasing \lce approximation~$\aset{a_s}\siome$, the Solovay test~$\aset{[a_s, a_{s+1}]}\siome$ is converging but does not cover any real.
\end{remark}

\begin{remark}\label{remark:test-give-approximations}
Consider a converging Solovay test that covers a real~$\gamma$. Then the sequence of endpoints of this test is computable and converges to~$\gamma$, hence is a \ca approximation of~$\gamma$. Similarly, if the Solovay test is \dcen, then the sequence of endpoints is a \dce approximation of~$\gamma$, and if the Solovay test is \lcen, then its sequence of left endpoints is a \lce approximation of~$\gamma$.   
\end{remark}

\begin{proposition}\label{prop:various-solovay-tests}
$\;$
\begin{itemize}
\item[(i)] A real is nonrandom and \ca if and only if it is covered by a  converging Solovay test.
\item[(ii)] A real is nonrandom and \dce if and only if it is covered by a \dce Solovay test.
\item[(iii)] A real is nonrandom and \lce if and only if it is covered by a \lce Solovay test.
\end{itemize}
\end{proposition}
\begin{proof}
For all three equivalences, the right-to-left direction is immediate by Remark~\ref{remark:test-give-approximations} and since being nonrandom is equivalent to being covered by a Solovay test. 

For a proof of the left-to-right direction of $(i)$, let the real~$\alpha$ be nonrandom and \ca Let~$\aset{I_i}\iiome$ be a Solovay test that covers~$\alpha$ and let~$\aset{a_i}\iiome$ be a \ca approximation of~$\alpha$. We can assume that~$\alpha$ is not rational, since otherwise there is nothing to prove. We can assume further that the intervals~$I_i$ are pairwise distinct. We define a sequence of intervals by a construction in stages~$s=0, 1, \ldots$. Initially, call all intervals~$I_i$ unused. At stage~$s$, if there is some index~$i \le s$ such that~$a_s$ is in the currently unused interval~$I_i$, append the interval with the least such index to the constructed sequence and declare this interval to be used. The constructed sequence of intervals is indeed a Solovay test since the construction is effective, the sequence contains only intervals from the given Solovay test, and each such interval is appended at most once. The constructed Solovay test covers~$\alpha$ because every interval~$I_i$ that contains~$\alpha$ also contains almost all~$a_s$, thus it follows by an easy inductive argument that every such intervals will eventually be appended to the constructed sequence. Finally, the constructed Solovay test is converging, since the~$a_s$ converge and the length of the intervals that are appended during some stage~$s$ tends to~$0$ when~$s$ tends to infinity.   

In the proof of the left-to-right direction of $(ii)$, we apply the same construction as above. However, now we choose~$\aset{a_i}\iiome$ as a \dce approximation of~$\alpha$. Let the constructed sequence of intervals be~$\aset{J_i}\iiome$, where~$J_i = [\ell_i , r_i]$.  We have already seen that~$\aset{J_i}\iiome$ is a Solovay test that covers~$\alpha$, so it suffices to show that this test is  actually \dce Let~$s_i$ be the stage at which~$J_i$ is appended, hence~$a_{s_i}$ is a member of~$J_i$ for all~$i$. Consequently, we have 
\[
|\ell_{i+1} - r_ i| = |\min J_{i+1} - \max J_{i}| \le  |a_{s_{i+1}} -a_{s_i}| + |J_i| + |J_{i+1}| ,
\]
which implies
\begin{equation}\label{eq:upper-bound-dce-test}
\sum_{\iiome} \underbrace{|r_i - \ell_i |}_{= |J_i|} + |\ell_{i+1} - r_ i| 
\le  3  \sum_{\iiome} |J_i| + \sum_{\iiome}|a_{s_{i+1}} -a_{s_i}| \le  3  \sum_{\iiome} |J_i| + \sum_{\iiome}|a_{i+1} -a_i| .
\end{equation}
Note that the last inequality in~\eqref{eq:upper-bound-dce-test} holds by the triangle inequality.
Since~$\aset{J_i}\iiome$ is a Solovay test, and the sequence~$\aset{a_i}\iiome$ has bounded variation, the upper bound in~\eqref{eq:upper-bound-dce-test} is finite.
As a consequence, the sequence of endpoints~$r_0, \ell_0, r_1, \ell_1, \ldots$ of the Solovay test~$\aset{J_i}\iiome$ has finite variation, hence the test is \dce

In the proof of the left-to-right direction of $(iii)$, we apply a variant of the same construction, where now we choose~$\aset{a_i}\iiome$ as a \lce approximation of~$\alpha$. The only difference is that whenever in the previous construction during some stage~$s$ some interval~$I$ would be appended, we append instead the interval~$I  \cap [a_s, 1]$. Observe that the latter interval contains~$\alpha$ if and only if~$I$ contains~$\alpha$ since we have~$a_s < \alpha < 1$. Consequently, the construction yields again a Solovay test that covers~$\alpha$, which is obviously \lce
\end{proof}

\begin{proposition}
A real is \lce speedable if and only if it is covered by a Solovay test that has bounded increments. 
\end{proposition}

\begin{proof}
Let~$\alpha$ be a \lce speedable real.
Then~$\alpha$ is strongly \lce $\frac{1}{3}$-speedable by Theorem~\ref{thm:equilce}.
Let $\aset{a_s}\siome$ be a \lce approximation of~$\alpha$ which is $\frac{1}{3}$-speedable via the function $s \mapsto s+1$.
Then for infinitely many~$s$, $\alpha - a_{s+1} \le \frac13(\alpha -a_s)$.
If we let~$I_s=[a_s,a_s + 2(a_{s+1} - a_s)]$, then~$\alpha\in I_s$ for each such~$s$, i.e., $\alpha$ is covered by the test~$\aset{I_s}\siome$. This test has finite measure $2(\alpha-a_0)$, hence is a Solovay test, and by construction has bounded increments with constant~$d=\frac{1}{2}$.

For the converse, let~$\alpha$ be covered by some Solovay test $\aset{[\ell_i,r_i]}\iiome $ that has bounded increments for some constant~$d>0$. Then the nondecreasing sequence~$\ell_0, \ell_1, \ldots$ converges to~$\alpha$, hence is a \lce approximation to~$\alpha$. Let~$i$ be any of the infinitely many indices such that $\alpha\in [\ell_i,r_i]$. We have~$\ell_i \le \ell_{i+1} \le \alpha \le r_i$ and~$\ell_{i+1} - \ell_i \ge d (r_i - \ell_i) \geq d(\alpha - \ell_{i})$, which implies~$d\le1$ and  

\[
\frac{\alpha - \ell_{i+1}}{\alpha - \ell_{i}} \leq \frac{\alpha - (\ell_i + d(r_i - \ell_i))}{\alpha - \ell_{i}} = 1-d .
\]

Consequently, $\aset{\ell_i}\iiome$ is $(1-d)$-speedable, hence~$\alpha$ is \lce speedable. 
\end{proof}
\section*{Acknowledgements}
Barmpalias was supported by Beijing Natural Science Foundation grant No.\ IS24013
Fang is supported by Beijing Natural Science Foundation grant No.\ 1262022 and was supported by NSFC grant No.\ 12001519.
Merkle and Titov were supported by DFG project 556436876.
Titov was also supported by the ANR project FLITTLA ANR-21-CE48-0023.
We would like to thank Liang Yu for helpful discussions. Furthermore, we are grateful to the anonymous referees for their careful review and suggestions.
\bibliographystyle{plainnat}
\bibliography{bibliography} 
\end{document}